\documentclass[12pt,a4paper]{amsart}
\usepackage[a4paper, margin=1in]{geometry}
\usepackage{amstext}
\usepackage{amsthm}
\usepackage[active]{srcltx}

\makeatletter

\pdfpageheight\paperheight
\pdfpagewidth\paperwidth

\usepackage{amsmath}

\usepackage{amssymb}
\usepackage{stmaryrd}
\usepackage{amsthm}
\usepackage[mathscr]{eucal}
\usepackage{amscd}
\usepackage[all]{xy}
\usepackage{url}

\newtheorem{Thm}{Theorem}[section]
\usepackage{aliascnt}

\usepackage{enumerate}

\usepackage{comment}

\usepackage{hyperref}
\hypersetup{%
pdftitle={Preprint},
pdfauthor={},
pdfkeywords={},
bookmarksnumbered,
pdfstartview={FitH},
breaklinks=true,
colorlinks=true,
urlcolor=blue,
citecolor=blue,
}%
\usepackage[all]{hypcap} 

\usepackage{breakurl}

\usepackage{color}

\newcommand{\red}[1]{{\color{red} #1}}

\theoremstyle{plain}
\newtheorem{Lem}[Thm]{Lemma}
\theoremstyle{plain}
\newtheorem{Prop}[Thm]{Proposition}
\theoremstyle{plain}
\newtheorem{Cor}[Thm]{Corollary}
\theoremstyle{plain}

\theoremstyle{plain}
\newtheorem*{Thm*}{Theorem}
\theoremstyle{plain}
\newtheorem*{Conj*}{Conjecture}

\theoremstyle{definition}
\newtheorem{Def}[Thm]{Definition}
\theoremstyle{definition}
\newtheorem*{Def*}{Definition}
\theoremstyle{definition}
\newtheorem{Eg}[Thm]{Example}
\theoremstyle{definition}

\theoremstyle{definition}
\newtheorem{Assumption}{Assumption}
\theoremstyle{definition}

\theoremstyle{definition}

\theoremstyle{definition}

\theoremstyle{definition}

\theoremstyle{remark}
\newtheorem{Rem}[Thm]{Remark}

\newcommand{\kk}{\Bbbk}
\newcommand{\Z}{\mathbb{Z}}
\newcommand{\N}{\mathbb{N}}
\newcommand{\Q}{\mathbb{Q}}

\renewcommand{\tilde}[1]{\widetilde{#1}}

\newcommand{\opname}[1]{\operatorname{\mathsf{#1}}}

\newcommand{\inj}{\opname{inj}}

\renewcommand{\deg}{\opname{deg}}

\newcommand{\Hf}{{\frac{1}{2}}}

\newcommand{\Rm}[1]{{\longmapsto}}
\newcommand{\Lm}[1]{{\longmapsfrom}}

\newcommand{\cA}{{\mathcal A}}

\newcommand{\cF}{{\mathcal F}}

\newcommand{\cM}{{\mathcal M}}

\newcommand{\cP}{{\mathcal P}}

\newcommand{\cU}{{\mathcal U}}

\newcommand{\bA}{{\mathbf A}}

\newcommand{\uk}{{\underline{k}}}

\newcommand{\tB}{{\widetilde{B}}}

\newcommand{\can}{L}
\newcommand{\clAlg}{{\cA}}

\newcommand{\tw}{{\tilde{w}}}

\usepackage{tikz}
\usetikzlibrary{positioning,shapes,shadows,arrows,snakes}
\usetikzlibrary {arrows.meta,positioning}

\pgfdeclarelayer{edgelayer}
\pgfdeclarelayer{nodelayer}
\pgfsetlayers{edgelayer,nodelayer,main}

\tikzstyle{none}=[inner sep=0pt]
\tikzstyle{black box}=[draw=black, fill=black!25]
\tikzstyle{white box}=[draw=black, fill=white]
\tikzstyle{black circle}=[circle,draw=black!50, fill=black!25]
\tikzstyle{red circle}=[circle,draw=red!50, fill=red!25]
\tikzstyle{blue circle}=[circle,draw=blue!50, fill=blue!25]
\tikzstyle{green circle}=[circle,draw=green!50, fill=green!25]
\tikzstyle{yellow circle}=[circle,draw=yellow!50, fill=yellow!25]

\newcommand{\thistheoremname}{}
\newtheorem*{genericthm*}{\thistheoremname}
\newenvironment{namedthm*}[1]
  {\renewcommand{\thistheoremname}{#1}%
   \begin{genericthm*}}
  {\end{genericthm*}}

\renewcommand{\inj}{{\bI}}
\renewcommand{\can}{{\bL}}

\newcommand{\fv}{\opname{f}}
\newcommand{\ufv}{\opname{uf}}

\makeatother

\begin{document}
\newcommand{\frg}{\mathfrak{g}}
\newcommand{\dCan}{\mathbf{B}^*}
\renewcommand{\bA}{\mathbb{A}}
\renewcommand{\can}{\mathbf{L}}
\renewcommand{\inj}{\mathbf{I}}

\newcommand{\bClAlg}{\overline{\clAlg}}
\newcommand{\LP}{\mathcal{LP}}
\newcommand{\hLP}{\widehat{\LP}}

\renewcommand{\Mc}{M^{\circ}}
\newcommand{\tropSet}{\mathcal{M}^\circ}

\newcommand{\upClAlg}{\cU}
\newcommand{\mm}{\mathbf{m}}
\newcommand{\urtriangle}{%
\begin{tikzpicture}%
\draw (1.5ex,0) -- (0,1.5ex);
\draw (1.5ex,1.5ex) -- (1.5ex,0);
\draw (0,1.5ex) -- (1.5ex,1.5ex);
\end{tikzpicture}%
}
\newcommand{\dt}{\tilde{t}}
\renewcommand{\tw}{\opname{tw}}
\title[]{Notes: From Dual Canonical Bases to Triangular Bases of Quantum Cluster
Algebras}
\author{Fan QIN}
\email{qin.fan.math@gmail.com}
\begin{abstract}
These notes are mainly based on \cite{qin2020dual} and a series of
talks given in the workshop \href{http://qsms.math.snu.ac.kr/CARTEA2023}{\color{blue}{\underline{CARTEA}}}.

For any symmetrizable Kac-Moody algebra $\frg$ and any Weyl group
element $w$, the corresponding quantum unipotent subgroup $A_{q}[N_{-}(w)]$
possesses the dual canonical basis $\dCan$. We show that the dual
canonical basis is the (common) triangular basis of the quantum cluster
algebra. Consequently, we deduce that the basis contains all quantum
cluster monomials, extending previous results by the author and Kang-Kashiwara-Kim-Oh. 
\end{abstract}

\maketitle
\tableofcontents{}

\section{Overview}

We fix the base ring $\kk=\Z[q^{\pm\Hf}]$ (at the classical level,
we can take $\kk=\Z,\Q,\ldots$).

Fomin and Zelevinsky \cite{fomin2002cluster} introduced cluster algebras
$\cA$ and expected:
\begin{itemize}
\item for many varieties $\bA$ from Lie theory, their (quantized) coordinate
rings $\kk[\bA]=\cA$,
\item $\kk[\bA]$ has a basis $\can$ in analogous to the dual canonical
basis $\dCan$ of quantum groups $\kk[N_{-}]$ and, moreover, $\can$
contains all cluster monomials.
\end{itemize}
It is natural to ask if the original dual canonical basis $\dCan$
meets the above expectations. By \cite{GeissLeclercSchroeer11}\cite{GY13,goodearl2020integral},
for any Kac-Moody algebra $\frg$ and any Weyl group element $w\in W$,
the quantum unipotent subgroup $\kk[N_{-}(w)]$ is isomorphic to some
$\bClAlg$, where $\bClAlg$ is a quantum cluster algebra whose \emph{frozen
variables }are not inverted.

\begin{Thm*}[{Fomin-Zelevinsky conjecture, main result \cite{qin2020dual}}]

$q^{\frac{\Z}{2}}\dCan$ of $\kk[N_{-}(w)]$ contains all quantum
cluster monomials.

\end{Thm*}

\begin{Rem}

\cite{qin2017triangular} verified the statement for $\frg$ of type
$ADE$ and partially for symmetric Kac-Moody algebras $\frg$ (i.e.,
the generalized Cartan matrix is symmetric). Both \cite{qin2020dual}
and \cite{qin2017triangular} are based on the \textbf{common triangular
bases} of cluster algebras.

\cite{Kang2018} verified the claim for all symmetric Kac-Moody $\frg$
based on a very different approach (monoidal categorification by quiver
Hecke algebras). After the appearance of \cite{qin2020dual}, \cite{mcnamara2021cluster}
generalized the result of \cite{Kang2018} to all $\frg$ as well.

\end{Rem}

Our approach for this result can be sketched as the following.

We consider the common triangular basis $\can$ in the sense of \cite{qin2017triangular}
which contains all cluster monomials.

\begin{Thm}\label{thm:existence}

Under some conditions, $\can$ exists.

\end{Thm}

\begin{Thm}\label{thm:quantum_groups_bases}

Theorem \ref{thm:existence} applies to $\kk[N_{-}(w)]$.

\end{Thm}

For $\kk[N_{-}(w)]$,~ $\dCan$ provides us $\can$, and the main
result is a direct consequence of Theorem \ref{thm:quantum_groups_bases}.

\begin{Rem}

$\can$ generalizes the dual canonical basis $\dCan$ to cluster algebras.
In addition, $\can$ exists for more (quantized) coordinate ring $\kk[\bA]$,
such as double Bruhat cells (in preparation). 

$\can$ is a Kazhdan-Lusztig type basis, see Remark \ref{rem:KL_basis}. 

In known cases where a cluster algebra $\clAlg$ is categorified by
a monoidal category $\cM$ \cite{HernandezLeclerc09}, that is, $\clAlg\simeq K_{0}(\cM)$
and the cluster monomials correspond to simple objects, the set of
simple objects provides us the common triangular basis $\can$.

There are no known examples such that the common triangular basis
$\can$ does not exist (under the assumption that its definition makes
sense, see Assumption \ref{assumption:injective-reachable}).

\end{Rem}

\begin{Rem}[Obstruction]

Unlike \cite{qin2017triangular}, we do NOT require $\can$ to be
positive (i.e., the structure constants are not necessarily positive).
In fact, it is known that $\dCan$ is not positive when $\frg$ is
not symmetric.

We rely on \textbf{tropical properties }instead of positivity, by
which we mean how Laurent degrees change under change of coordinates
(\emph{mutations}).

\end{Rem}

\section*{Acknowledgements}
These notes are based on the series of talks given by the author in the workshop \href{http://qsms.math.snu.ac.kr/CARTEA2023}{\color{blue}{\underline{CARTEA}}}, where the approach to the main result of \cite{qin2020dual} was presented. He deeply
thanks the organizers for their invitation and their kind hospitality.

\section{Cluster algebras}

Let $I=I_{\ufv}\sqcup I_{\fv}$ denote a finite set of vertices with
a partition into \textbf{unfrozen and frozen} ones. 

A \textbf{seed} $t$ is a collection $((X_{i}(t))_{i\in I},\tB(t),\Lambda(t))$,
where $X_{i}(t)$ are cluster variables, $\tB(t)$ an $I_{\ufv}\times I$
$\Z$-matrix, $\Lambda(t)$ an $I\times I$ skew-symmetric $\Z$-matrix,
such that $(\tB(t),\Lambda(t))$ is a compatible pair in the sense
of \cite{BerensteinZelevinsky05}. Denote the matrix entries by $b_{ij}(t)$
and $\Lambda_{ij}(t)$ respectively.

We often omit the symbol $t$ for simplicity.

Let $\cdot$ denote the \textbf{commutative product}: $X_{i}\cdot X_{j}=X_{j}\cdot X_{i}=:X_{i}X_{j}$.

Unless otherwise specified, we use the \textbf{twisted product} $*$
such that
\begin{align*}
X_{i}*X_{j} & :=q^{\Hf\Lambda_{ij}}X_{i}\cdot X_{j}.
\end{align*}

Define the Laurent polynomial ring $\LP(t):=\kk[X_{i}(t)^{\pm}]_{i\in I}$
and denote its fraction field by $\cF(t)$.

Denote $\Mc(t)=\oplus_{i\in I}\Z f_{i}=\Z^{I}$, where $f_{i}$ is
the $i$-th unit vector. For any $m=\sum m_{i}f_{i}\in\Mc(t)$, define
$X^{m}:=\prod X_{i}^{m_{i}}$ where we use the $\cdot$ product. The
\textbf{cluster monomials }in $t$ are $X^{m}$ with $m\geq0$.

Denote $N_{\ufv}(t)=\oplus_{k\in I_{\ufv}}\Z e_{k}=\Z^{I_{\ufv}}$,
where $e_{k}$ is the $k$-th unit vector. For any $n=\sum n_{k}e_{k}\in N_{\ufv}(t)$,
define $Y^{n}:=X^{\tB n}$.

For any $k\in I_{\ufv}$, the mutation $\mu_{k}$ generates a new
seed $t':=\mu_{k}t$ such that

\begin{align*}
X_{i}' & =\begin{cases}
X_{i} & i\neq k\\
X^{-f_{k}+\sum_{i}[b_{ik}]_{+}f_{i}}+X^{-f_{k}+\sum_{j}[-b_{jk}]_{+}f_{j}} & i=k
\end{cases},
\end{align*}
where $[\ ]_{+}:=\max(\ ,0)$. See \cite{BerensteinZelevinsky05}
for $\tB'$ and $\Lambda'$. Note that the \textbf{frozen variables}
$X_{j}$, $j\in I_{f}$, remain unchanged.

In fact, the above formula gives rise to an isomorphism $\mu_{t',t}^{*}:\cF(t')\simeq\cF(t)$,
but we will usually omit this symbol by identifying $\cF(t')$ and
$\cF(t)$ for simplicity.

Given any initial seed $t_{0}$. We denote 
\begin{align*}
\Delta^{+} & :=\Delta_{t_{0}}^{+}:=\{\text{seeds obtained from \ensuremath{t_{0}} by iterated mutations}\}
\end{align*}
Define:
\begin{itemize}
\item The (partially compactified) cluster algebra $\bClAlg:=\kk[X_{i}(t)]_{\forall i,\forall t\in\Delta^{+}}$.
\item The (localized) cluster algebra $\clAlg:=\bClAlg[X_{j}^{-1}]_{j\in I_{\fv}}$.
\item The upper cluster algebra $\upClAlg:=\cap_{t\in\Delta^{+}}\LP(t)$.
We have $\clAlg\subset\upClAlg$ by \cite{fomin2002cluster}\cite{BerensteinZelevinsky05},
and often $\clAlg=\upClAlg$.
\end{itemize}
We sometimes denote $\upClAlg$ by $\upClAlg(t)$ when we want to
view it as a subalgebra of $\LP(t)$.

\begin{Eg}\label{eg:SL_3}

We take $\kk=\Z$, $G=SL_{3}$, and
\[
N_{-}:=\left\{ g=\left(\begin{array}{ccc}
1 & 0 & 0\\
X_{1} & 1 & 0\\
X_{2} & X_{1}' & 1
\end{array}\right),\text{\ensuremath{\forall X_{1},X_{1}',X_{2}} }\right\} \subset G.
\]

Then $\kk[N_{-}]=\kk[X_{1},X_{1}',X_{2}]$. Define the minor $X_{3}:=X_{1}X_{1}'-X_{2}$.
Denote $I=I_{\ufv}\sqcup I_{\fv}:=\{1\}\sqcup\{2,3\}$. The initial
seed $t$ consists of the cluster variables $X_{i}(t):=X_{i}$, $\forall i$,
and the matrix $\tB(t)=\left(\begin{array}{c}
0\\
1\\
-1
\end{array}\right)$. The mutated seed $t':=\mu_{1}(t)$ has a new cluster variable $X_{1}(t'):=X_{1}'$
and the new matrix $\tB(t')=\left(\begin{array}{c}
0\\
-1\\
1
\end{array}\right)$. We have $\bClAlg=\kk[N_{-}]$ and $\upClAlg=\clAlg=\bClAlg[X_{2}^{-1},X_{3}^{-1}]$.

See \cite[Example 2.2]{qin2021cluster} for the quantization.

\end{Eg}

\section{Triangular bases for one seed}

Choose and fix any $t\in\Delta^{+}$. 

Any $z\in\LP(t)$ of the form
\begin{align*}
z= & X^{g}\cdot(1+\sum_{0<n\in\N^{I_{\ufv}}}c_{n}Y^{n})
\end{align*}
for some $g\in\Mc(t)=\Z^{I}$ is called \textbf{$g$-pointed}. Define
the \textbf{degree }$\deg^{t}z:=g$ and the \textbf{normalization
}$[q^{\alpha}z]^{t}:=z$ for any $\alpha\in\Hf\Z$.

We introduce the \textbf{dominance order} $\prec_{t}$ on $\Mc(t)$
such that

\begin{align*}
\deg^{t}X^{g}Y^{n} & \prec_{t}\deg^{t}X^{g}
\end{align*}
whenever $0<n\in\N^{I_{\ufv}}$.

\begin{Rem}

$z$ is analogous to the character of a highest weight module $S(w)$
(for example, of a quantum affine algebra):

\begin{align*}
\chi(S(w)) & =e^{w}\cdot(1+\sum_{v>0}c_{v}e^{v}).
\end{align*}

\end{Rem}

From now on, we make the following assumption (equivalently, there
exists a green to red sequence \cite{keller2011cluster}):

\begin{Assumption}[Injective-reachable]\label{assumption:injective-reachable}

$\exists$ a seed $t[1]$ and a permutation $\sigma$ of $I_{\ufv}$
such that, $\forall k\in I_{\ufv}$,

\begin{align*}
\deg^{t}X_{\sigma k}(t[1]) & \equiv-f_{k}\text{ mod }\Z^{I_{\fv}}.
\end{align*}

\end{Assumption}

Denote $I_{k}(t):=X_{\sigma k}(t[1])$, called the\textbf{ $k$}-th
\textbf{injective cluster variable}.\footnote{It can be categorified by the $k$-th injective module of a Jacobian
algebra \cite{DerksenWeymanZelevinsky09}.}

\begin{Rem}

The assumption is satisfied by almost all well-known cluster algebras.

If $\clAlg$ is categorified by a triangular category (cluster category),
there are rigid objects $T_{k}$ such that 

\begin{align*}
T_{k} & \xrightarrow{\text{categorify}}X_{k}(t)\\
T_{k}[1] & \xrightarrow{\text{categorify}}X_{\sigma k}(t[1])
\end{align*}
where $[1]$ is the suspension functor.

If $\clAlg$ is categorified by a monoidal category \cite{HernandezLeclerc09},
there are simple objects $S_{k}$ such that 

\begin{align*}
S_{k} & \xrightarrow{\text{categorify}}X_{k}(t)\\
D(S_{k}) & \xrightarrow{\text{categorify}}X_{\sigma k}(t[1])
\end{align*}
where the functor $D(\ )$ is the right dual (or the left dual, depending
on the convention).

\end{Rem}

\begin{Def}[{Triangular basis \cite{qin2017triangular}}]

A $\kk$-basis $\can$ of $\upClAlg(t)$ is called a \textbf{triangular
basis}\footnote{When $t$ is \emph{acyclic}, a different definition of triangular
bases was given by \cite{BerensteinZelevinsky2012}. It turns out
that the two definitions are equivalent for acyclic seeds \cite{qin2019compare}\cite{qin2020dual}.} with respect to $t$ if
\begin{itemize}
\item $\can$ contains all cluster monomials in $t,t[1]$.
\item $\can=\{\can_{g}|g\in\Mc(t)=\Z^{I}\}$ such that $\can_{g}$ is $g$-pointed.
\item $\overline{\can_{g}}=\can_{g}$, where the \textbf{bar involution}
$\overline{(\ )}$ on $\LP(t)$ is defined by $\overline{q^{\alpha}X^{m}}:=q^{-\alpha}X^{m}$.
\item $\forall i\in I,g\in\Mc(t)$, we have
\begin{align}
[X_{i}(t)*\can_{g}]^{t} & =\can_{g+f_{i}}+\sum_{g'\prec_{t}g+f_{i}}b_{g'}\can_{g'}\label{eq:triangularity}
\end{align}
such that $b_{g'}\in q^{-\Hf}\Z[q^{-\Hf}]=:\mm$.
\end{itemize}
\end{Def}

\begin{Eg}

In Example \ref{eg:SL_3}, we have $Y_{1}=X^{f_{2}-f_{3}}$ and thus
$X_{1}'=X^{-f_{1}+f_{3}}\cdot(1+Y_{1})$. So we have $t'=t[1]$ with
the trivial permutation $\sigma$. Choose a quantization, see \cite[Example 2.2]{qin2021cluster}.
The triangular basis with respect to $t$ is the set of the localized
quantum cluster monomials, i.e., $X^{m}(t)$ and $X^{m}(t')$ for
$m\in\N^{I_{\ufv}}\oplus\Z^{I_{\fv}}$.

\end{Eg}

\section{Uniqueness}

Assume $\can$ is a triangular basis with respect to $t$. We say
RHS of \eqref{eq:triangularity} is a\textbf{ $(\prec_{t},\mm)$-decomposition}
of LHS. Without requiring $b_{g'}\in\mm$, RHS is called a \textbf{$\prec_{t}$-decomposition}.
A $\prec_{t}$-decomposition into pointed elements is unique \cite{qin2019bases}
and convergent \cite{davison2019strong} in the completion 
\begin{align*}
\hLP(t) & :=\LP(t)\otimes_{\kk[y^{n}]_{n\in\N^{I_{\ufv}}}}\kk\llbracket y^{n}\rrbracket_{n\in\N^{I_{\ufv}}}.
\end{align*}
We denote \eqref{eq:triangularity} by $\text{LHS}\urtriangle_{(\prec_{t},\mm)}\can$
for simplicity.

We introduce the set of \textbf{distinguished functions} $\inj:=\{\inj_{g}|g\in\Mc(t)=\Z^{I}\}$
with respect to $t$, such that $\inj_{g}$ is $g$-pointed and of
the form 
\begin{align*}
\inj_{g} & :=[p_{g}*X^{[g_{\ufv}]_{+}}*I^{[-g_{\ufv}]_{+}}]^{t},
\end{align*}
where $p_{g}$ is a \textbf{frozen factor}, i.e., $p_{g}=X^{m}$ for
some $m\in\Z^{I_{\fv}}$, $g_{\ufv}$ is the natural projection of
$g$ in $\Z^{I_{\ufv}}$. By definition of the triangular basis, we
have $\inj_{g}\urtriangle_{(\prec_{t},\mm)}\can$ for all $g$. We
denote $\inj\urtriangle_{(\prec_{t},\mm)}\can$ in this case.

\begin{Lem}

For any $g\in\Mc(t)=\Z^{I}$, there exists a unique function $f_{g}$
in $\hLP(t)$ such that $f_{g}\urtriangle_{(\prec_{t},\mm)}\can$. 

\end{Lem}

$f_{g}$ are called the triangular functions (with respect to the
distinguished functions). 

\begin{Rem}\label{rem:KL_basis}

$f_{g}$ is computed from $\inj$ by an infinite step version of the
recursive algorithm for constructing \textbf{Kazhdan-Lusztig bases}.
The algorithm is infinite because $\prec_{t}$ is unbounded on $\Z^{I}$.

A priori, it is a formal Laurent series and depends on the choice
of $t$.

In general, $\inj=\{\inj_{g}|g\in\Z^{I}\}$ is NOT a basis.

\end{Rem}

If the triangular functions form a basis of $\upClAlg(t)$, it is
called the \textbf{weakly triangular basis} with respect to $t$.
Note that the triangular basis $\can$ must be weakly triangular.
So we obtain the following result.

\begin{Lem}

The triangular basis $\can$ with respect to $t$ is unique if it
exists.

\end{Lem}

\section{Tropical transformations, common triangular basis}

For any $t,t'\in\Delta^{+}$, we have some piecewise linear bijective
map \cite{FockGoncharov09}, called the \textbf{tropical transformation}:

\begin{align*}
\phi_{t',t} & :\Mc(t)\simeq\Mc(t').
\end{align*}
For any $g\in\Mc(t)$, we denote $g':=\phi_{t',t}g$. It is known
that $\phi_{t_{3},t_{2}}\phi_{t_{2},t_{1}}=\phi_{t_{3},t_{1}}$.

For any $z\in\LP(t)\cap\LP(t')$, we say $z$ is compatibly pointed
at $t,t'$ if $\exists g\in\Mc(t)$, such that $z$ is $g$-pointed
in $\LP(t)$ and $g'$-pointed in $\LP(t')$. We say a collection
$\{z_{1},z_{2},\ldots\}$ is compatibly pointed at $t,t'$ if each
element $z_{i}$ is.

We introduce the set of tropical points 
\begin{align*}
\tropSet & :=\sqcup_{t\in\Delta^{+}}\Mc(t)/\sim
\end{align*}
where the equivalence relation $\sim$ is generated by the bijections
$\phi$. Let $[g]$ denote the equivalence class of $g$.

We say $z\in\upClAlg$ is $[g]$-pointed if it is pointed in each
$\LP(t)$ and $[\deg^{t}z]=[g]$ for any $t$.

\begin{Lem}[{\cite{qin2017triangular}}]

For any $k\in I_{\ufv}$, the set of distinguished functions $\inj$
with respect to $t$ is compatibly pointed at $t,\mu_{k}t$.

\end{Lem}

For our application to quantum groups, we need to know the following
result hold for $\kk[N_{-}(w)]$, which was verified by \cite{Demonet10}.
General cases were proved in \cite{gross2018canonical}.

\begin{Thm}

Cluster monomials are $[g]$-pointed at distinct $[g]$.

\end{Thm}

\begin{Def}

If a $\kk$-basis $\can$ of $\upClAlg$ is the triangular basis with
respect to all $t\in\Delta^{+}$, and $\can=\{\can_{[g]}|[g]\in\tropSet\}$
such that $\can_{[g]}$ are $[g]$-pointed, then $\can$ is called
the \textbf{common triangular basis}.

\end{Def}

\section{Admissibility and compatibility}

We will show how a triangular basis is propagated on adjacent seeds.
Choose and fix any $t\in\Delta^{+}$ and $t'=\mu_{k}t$.

\begin{Def}

The triangular basis $\can^{t}$ with respect to $t$ is admissible
in direction $k$ if it contains the cluster monomials $X_{k}(t')^{d},I_{k}(t')^{d}$,
$d\in\N$.

\end{Def}

\begin{Prop}[{Admissibility implies compatibility \cite[Proposition 6.4.3]{qin2020dual}}]\label{prop:admissibility_to_compatibility}

If the triangular basis $\can^{t}$ is admissible in direction $k$,
then the following claims are true.
\begin{itemize}
\item $\can^{t}$ is compatibly pointed at $t,t'$.
\item $\can^{t}$ is the triangular basis $\can^{t'}$ with respec to $t'$.
\end{itemize}
\end{Prop}

\begin{proof}

(i) In $\LP(t')$, we have 
\begin{align*}
\inj_{g'}^{t'} & =\begin{cases}
[p_{g}*(X')^{d_{X}}*(X_{k}')^{g_{k}'}*(I')^{d_{I}}]^{t'} & g_{k}'\geq0\\{}
[p_{g}*(X')^{d_{X}}*(I_{k}')^{-g_{k}'}*(I')^{d_{I}}]^{t'} & g_{k}'<0
\end{cases},
\end{align*}
where $d_{X},d_{I}\in\N^{I\backslash\{k\}}$. In $\LP(t)$, by the
\textbf{sign-coherence }of degrees, one can check that the normalization
$[\ ]^{t'}$ above equals $[\ ]^{t}$, which implies the following
in $\LP(t)$:
\begin{align*}
\inj_{g'}^{t'} & =\begin{cases}
[p_{g}*(X)^{d_{X}}*\red{(X_{k}')^{g_{k}'}}*(I)^{d_{I}}]^{\red{t}} & g_{k}'\geq0\\{}
[p_{g}*(X)^{d_{X}}*\red{(I_{k}')^{-g_{k}'}}*(I)^{d_{I}}]^{\red{t}} & g_{k}'<0
\end{cases}.
\end{align*}
Admissibility implies that the factors $\red{(X_{k}')^{g_{k}'}}$
and $\red{(I_{k}')^{-g_{k}'}}$ are contained in $\can^{t}$, thus
they have $(\prec_{t},\mm)$-decomposition into $\inj^{t}$. Substituting
these factors by their decomposition and using the triangularity $[X^{m}*I^{\tilde{m}}]^{t}\urtriangle_{(\prec_{t},\mm)}\can^{t}\urtriangle_{(\prec_{t},\mm)}\inj^{t}$
for any $m,\tilde{m}\in\N^{I_{\ufv}}$, we can deduce $\inj^{t'}\urtriangle_{(\prec_{t},\mm)}\inj^{t}$,
see \cite[Lemma 6.2.3]{qin2017triangular}.

Combined with $\inj^{t}\urtriangle_{(\prec_{t},\mm)}\can^{t}$, we
get $\inj^{t'}\urtriangle_{(\prec_{t},\mm)}\can^{t}$. Taking the
inverse, we get $\can^{t}\urtriangle_{(\prec_{t},\mm)}\inj^{t'}$.
So we obtain the following $(\prec_{t},\mm)$-decomposition in $\hLP(t)$
for any $\can_{g}^{t}$:

\begin{align}
\can_{g}^{t} & =\inj_{g'}^{t'}+\sum_{\eta\prec_{t}g}b_{g,\eta}\inj_{\eta'}^{t'}.\label{eq:can_inj_decomposition}
\end{align}
where $b_{g,\eta}\in\mm$. (Recall that $\eta'=\phi_{t',t}\eta$.)

By \cite{qin2019bases}, since $\inj^{t'}$ is compatibly pointed
at $t,t'$, \eqref{eq:can_inj_decomposition} is also a $\prec_{t'}$-decomposition
in $\hLP(t')$. By the bar-invariance of $\can_{g}^{t}$ in $\hLP(t')$,
its $\prec_{t'}$-leading term (i.e., the term with the highest $\prec_{t'}$
Laurent degree) can only be the contribution from $\inj_{g'}^{t'}$
whose coefficient is $1$, because the coefficients of all the other
terms are NOT bar-invariant. Therefore, it is a $(\prec_{t'},\mm)$-decomposition
such that the $\prec_{t'}$-leading term is $\inj_{g'}^{t'}$. 

In particular, $\can_{g}^{t}$ is $g'$-pointed in $\LP(t')$.

(ii) We have shown in (i) that $\can^{t}$ is weakly triangular with
respect to $t'$. Combining with the fact that $\can^{t}$ is triangular
with respect to $t$ and compatibly pointed at $t,t'$, we can deduce
that $\can^{t}$ is further triangular with respect to $t'$, see
\cite[Proposition 6.2.8]{qin2020dual}.

\end{proof}

\begin{Prop}[{Compatibility implies admissibility \cite[Proposition 6.4.5]{qin2020dual}}]\label{prop:compatibility_to_admissibility}

The following claims are true.
\begin{enumerate}
\item Assume $z\in\upClAlg$ is compatibly pointed at $t,t[-1]$, and $\deg^{t}z=\deg^{t}(X_{k}')^{d}$
for some $d\geq0$, then $z=(X_{k}')^{d}$.
\item Assume $z\in\upClAlg$ is compatibly pointed at $t[1],t$, and $\deg^{t}z=\deg^{t}(I_{k}')^{d}$
for some $d\geq0$, then $z=(I_{k}')^{d}$.
\end{enumerate}
\end{Prop}

\begin{proof}

By \cite{qin2019bases}, there exists a bijection $\psi:\Mc(t)\simeq\Mc(t[-1])$
such that $z$ is $\eta$-\textbf{copointed} in $\LP(t)$ (i.e., $z=X^{\eta}\cdot(1+\sum_{n<0}c_{n}Y^{n})$)
if and only if $z$ is $\psi(\eta)$-pointed in $\LP(t[-1])$.

(1) By the assumption, $z$ and $(X_{k}')^{d}$ are both $g$-pointed
and $\eta$-copointed for some $g,\eta$: if we write $(X_{k}')^{d}=X^{g}\cdot(1+\sum_{0<n<d}c_{n}Y_{k}^{n}+Y_{k}^{d})$,
then $z=X^{g}\cdot(1+\sum_{0<n<d}c'_{n}Y_{k}^{n}+Y_{k}^{d})$. One
can then deduce from $z\in\LP(t)\cap\LP(t')$ that $z=(X_{k}')^{d}$.

(2) The claim is deduced from (1) by replacing $t$ by $t[1]$.

\end{proof}

\begin{Cor}

Let $\can$ be the triangular basis with respect to some seed $t$.
Assume that one of the following is true:
\begin{enumerate}
\item $\can$ contains all cluster monomials.
\item $\can=\{\can_{[g]}|[g]\in\tropSet\}$ such that $\can_{[g]}$ are
$[g]$-pointed.
\end{enumerate}
Then $\can$ is the common triangular basis.

\end{Cor}

\section{Similar seeds}

Given any seed $t\in\Delta^{+}$. Let $\dt$ denote another seed (not
necessarily in $\Delta^{+}$) with $\tilde{I}=I_{\ufv}\sqcup\tilde{I}_{\fv}$
such that $I_{\ufv}$ is the set of unfrozen vertices. Let $\tilde{b}_{ij}$
denote the entries of $\tB(\dt)$.

\begin{Def}

We say $\dt$ and $t$ are similar if there exists a permutation $\tau$
on $I_{\ufv}$ such that $b_{ij}=\tilde{b}_{\tau i,\tau j}$ for all
$i,j\in I_{\ufv}$.

\end{Def}

Note that $t$ and $t[1]$ are similar. The permutation $\tau$ acts
on $\Z^{I_{\ufv}}$ linearly such that $\tau f_{i}:=f_{\tau i}$ where
$f_{i}$ denote the $i$-th unit vectors. The following result is
an easy consequence of the correction technique \cite{Qin12}\cite[Section 4.2]{qin2020dual}.

\begin{Lem}\label{lem:similar_triangular}

If $\upClAlg(t)$ has the triangular basis $\can^{t}$ with respect
to $t$, such that $\can_{g_{\ufv}}^{t}=X^{g_{\ufv}}\cdot(1+\sum_{n>0}c_{n}Y^{n})$
for $g_{\ufv}\in\Z^{I_{\ufv}}$, then $\upClAlg(\dt)$ has the triangular
basis $\can^{\dt}$, such that 
\begin{align*}
\can_{g_{\ufv}}^{\dt} & =X^{\tau g_{\ufv}}\cdot(1+\sum_{n>0}c_{n}Y^{\tau n}).
\end{align*}

\end{Lem}

For any sequence $\uk:=(k_{1},\dots,k_{r})$ of letters in $I_{\ufv}$,
denote the mutation sequence $\mu_{\uk}:=\mu_{k_{r}}\cdots\mu_{k_{1}}$
(read from right to left). We also define $\tau\uk:=(\tau k_{1},\ldots,\tau k_{r})$.
The following result follows from the separation formulas of cluster
monomials in \cite{FominZelevinsky07}\cite{Tran09}\cite{BerensteinZelevinsky05}.

\begin{Cor}\label{cor:similar_basis_admissibility}

Under the assumption of Lemma \ref{lem:similar_triangular}, for any
$d\in\N$, if $X_{k_{r}}(\mu_{\uk}t)^{d}\in\can^{t}$, then $X_{\tau k_{r}}(\mu_{\tau\uk}\dt)^{d}\in\can^{\dt}$.

\end{Cor}

\section{Existence}

Given some seed $t\in\Delta^{+}$. Assume there is a mutation sequence
$\uk=(k_{1},\ldots,k_{r})$ such that $t[1]=\mu_{\uk}t$ for some
permutation $\sigma$ (i.e., $I_{k}(t)=X_{\sigma k}(t[1])$). We denote
$t_{0}=t$ and $t_{i}=\mu_{k_{i}}\ldots\mu_{k_{1}}t_{0}$ for $i\in[1,r]$.
Then $t_{r}=t[1]$.

\begin{Assumption}

We make the following assumptions:
\begin{enumerate}
\item $\upClAlg$ has the triangular basis $\can^{t}$ with respect to $t$.
\item $\can^{t}$ contains the cluster monomials $X_{k_{i}}(t_{i})^{d}$
and $I_{k_{i}}(t_{i})^{d}$, $d\in\N$, $i\in[1,r]$, appearing along
the mutation sequence $\mu_{\uk}$.
\end{enumerate}
\end{Assumption}

\begin{Thm*}[{Theorem \ref{thm:existence}}]

Under the above assumptions, $\can^{t}$ is the common triangular
basis.

\end{Thm*}

\begin{proof}

By Proposition \ref{prop:admissibility_to_compatibility} (admissibility
implies compatibility), $\can^{t}=\can^{t_{1}}=\cdots=\can^{t[1]}$
is the triangular basis with respect to $t_{i}$ and compatibly pointed
at $t_{i}$ for all $i\in[0,r]$. Recall that $t[s]$ and $t[s+1]$
are similar for any $s\in\Z$. Applying Corollary \ref{cor:similar_basis_admissibility}
for the mutation sequence $\mu_{\sigma^{s}\uk}$ from $t[s]$ to $t[s+1]$
recursively for $s$, we get a chain of \textbf{compatible triangular
basis} 
\begin{align*}
\cdots & =\can^{t[-1]}=\can^{t}=\can^{t[1]}=\cdots,
\end{align*}
that is, $\can^{t}$ is the triangular basis with respect to $t[s]$
for any $s\in\Z$, and it is compatibly pointed at $t,t[s]$ $\forall s$.

Take any $j\in I_{\ufv}$. By Proposition \ref{prop:compatibility_to_admissibility}
(compatibility implies admissibility), the cluster monomials $X_{j}(\mu_{j}t)^{d}$
belong to $\can^{t}$, $d\in\N$. Since $t$ and $t[1]$ are similar,
Corollary \ref{cor:similar_basis_admissibility} implies $I_{j}(t)^{d}=X_{\sigma j}(\mu_{\sigma j}t[1])^{d}\in\can^{t[1]}=\can^{t}$.
So $\can^{t}$ is admissible in direction $j$. Applying Corollary
\ref{cor:similar_basis_admissibility} for the pair of similar seeds
$t$ and $t[s]$, we deduce that $\can^{t[s]}$ is admissible in direction
$\sigma^{s}j$. So we get two chains of compatible triangular bases:

\begin{align*}
 & \cdots &  & = &  & \can^{t[-1]} &  & = &  & \can^{t} &  & = &  & \can^{t[1]} & \cdots\\
 &  &  &  &  & \parallel &  &  &  & \parallel &  &  &  & \parallel\\
 & \cdots &  & = &  & \can^{(\mu_{j}t)[-1]} &  & = &  & \can^{\mu_{j}t} &  & = &  & \can^{(\mu_{j}t)[1]} & \cdots
\end{align*}

Repeat the above process, we get compatible triangular bases $\can^{t}=\can^{t'}$
at $t$ and any $t'\in\Delta^{+}$. 

\end{proof}

Take any quantum unipotent subgroup $\kk[N_{-}(w)]\simeq\bClAlg(t_{0})$,
where the initial seed $t_{0}$ is constructed from any given reduced
word of $w$. Define 
\begin{align*}
\can^{t_{0}} & :=\{[p*b]^{t_{0}}|p\text{ frozen factor},\ b\in \dCan\}.
\end{align*}
Then $\can^{t_{0}}$ is a basis of $\upClAlg(t_{0})=\clAlg(t_{0})=\kk[N_{-}(w)][\cP]$,
where $\cP$ is the multiplicative group generated by the frozen factors
and $q^{\Hf\Z}$.

\begin{Prop}[{\cite{qin2017triangular}}]

$\can^{t_{0}}$ is the triangular basis with respect to $t_{0}$.

\end{Prop}

\begin{Rem}

The proof works for Kazhdan-Lusztig type bases: the triangularity
between the standard basis and the canonical basis implies the $\prec_{t_{0}}$-triangularity.

\end{Rem}

\begin{Thm*}[{Theorem \ref{thm:quantum_groups_bases}}]

$\can^{t_{0}}$ is the common triangular basis of $\kk[N_{-}(w)][\cP]$.

\end{Thm*}

\begin{proof}

We know that there exists a sequence $\uk=(k_{1},\ldots,k_{r})$ such
that $t[1]=\mu_{\uk}t$ and $\can^{t_{0}}$ contains $X_{k_{i}}(t_{i})^{d}$,
$i\in[1,r]$ (called $T$-systems of unipotent quantum minors \cite{GeissLeclercSchroeer11}\cite{GY13})
.

By \cite{kimura2017twist}, there exists an automorphism $\tw$ of
$\clAlg$, called the twist automorphism, such that $\tw(\can^{t_{0}})=\can^{t_{0}}$. 

We further show that, $\forall i\in[1,r]$,

\begin{align*}
\tw(X_{k_{i}}(t_{i})) & =p_{i}\cdot I_{k_{i}}(t_{i}),\ \text{some frozen factor }p_{i}.
\end{align*}
It follows that $I_{k_{i}}(t_{i})^{d}\in\can^{t_{0}}$. Then the desired
claim follows from Theorem \ref{thm:existence}.

\end{proof}






\begin{thebibliography}{KKKO18}

\bibitem[BZ05]{BerensteinZelevinsky05}
Arkady Berenstein and Andrei Zelevinsky, \emph{Quantum cluster algebras}, Adv. Math. \textbf{195} (2005), no.~2, 405--455, \href {http://arxiv.org/abs/math/0404446v2} {\path{arXiv:math/0404446v2}}.

\bibitem[BZ14]{BerensteinZelevinsky2012}
\bysame, \emph{Triangular bases in quantum cluster algebras}, International Mathematics Research Notices \textbf{2014} (2014), no.~6, 1651--1688, \href {http://arxiv.org/abs/1206.3586} {\path{arXiv:1206.3586}}.

\bibitem[Dem10]{Demonet10}
Laurent Demonet, \emph{Mutations of group species with potential and their representations. {A}pplication to cluster algebras}, 2010, \href {http://arxiv.org/abs/1003.5078v1} {\path{arXiv:1003.5078v1}}.

\bibitem[DM21]{davison2019strong}
Ben Davison and Travis Mandel, \emph{Strong positivity for quantum theta bases of quantum cluster algebras}, Inventiones mathematicae (2021), 1--119, \href {http://arxiv.org/abs/1910.12915} {\path{arXiv:1910.12915}}.

\bibitem[DWZ10]{DerksenWeymanZelevinsky09}
Harm Derksen, Jerzy Weyman, and Andrei Zelevinsky, \emph{Quivers with potentials and their representations {II}: {Applications to cluster algebras}}, J. Amer. Math. Soc. \textbf{23} (2010), no.~3, 749--790, \href {http://arxiv.org/abs/0904.0676} {\path{arXiv:0904.0676}}.

\bibitem[FG09]{FockGoncharov09}
Vladimir~V. Fock and Alexander~B. Goncharov, \emph{Cluster ensembles, quantization and the dilogarithm}, Annales scientifiques de l'Ecole normale sup{\'e}rieure \textbf{42} (2009), no.~6, 865--930, \href {http://arxiv.org/abs/math.AG/0311245} {\path{arXiv:math.AG/0311245}}.

\bibitem[FZ02]{fomin2002cluster}
Sergey Fomin and Andrei Zelevinsky, \emph{Cluster algebras. {I}: {F}oundations}, Journal of the American Mathematical Society \textbf{15} (2002), no.~2, 497--529, \href {http://arxiv.org/abs/math/0104151} {\path{arXiv:math/0104151}}.

\bibitem[FZ07]{FominZelevinsky07}
\bysame, \emph{Cluster algebras {IV}: Coefficients}, Compositio Mathematica \textbf{143} (2007), 112--164, \href {http://arxiv.org/abs/math/0602259} {\path{arXiv:math/0602259}}.

\bibitem[GHKK18]{gross2018canonical}
Mark Gross, Paul Hacking, Sean Keel, and Maxim Kontsevich, \emph{Canonical bases for cluster algebras}, Journal of the American Mathematical Society \textbf{31} (2018), no.~2, 497--608, \href {http://arxiv.org/abs/1411.1394} {\path{arXiv:1411.1394}}.

\bibitem[GLS13]{GeissLeclercSchroeer11}
Christof Gei\ss, Bernard Leclerc, and Jan Schr{\"o}er, \emph{Cluster structures on quantum coordinate rings}, Selecta Mathematica \textbf{19} (2013), no.~2, 337--397, \href {http://arxiv.org/abs/1104.0531} {\path{arXiv:1104.0531}}.

\bibitem[GY16]{GY13}
K.R. Goodearl and M.T. Yakimov, \emph{Quantum cluster algebra structures on quantum nilpotent algebras}, Memoirs of the American Mathematical Society \textbf{247} (2016), no.~1169, \href {http://arxiv.org/abs/1309.7869} {\path{arXiv:1309.7869}}.

\bibitem[GY21]{goodearl2020integral}
\bysame, \emph{Integral quantum cluster structures}, Duke Mathematical Journal \textbf{170} (2021), no.~6, 1137--1200, \href {http://arxiv.org/abs/2003.04434} {\path{arXiv:2003.04434}}.

\bibitem[HL10]{HernandezLeclerc09}
David Hernandez and Bernard Leclerc, \emph{Cluster algebras and quantum affine algebras}, Duke Math. J. \textbf{154} (2010), no.~2, 265--341, \href {http://arxiv.org/abs/0903.1452} {\path{arXiv:0903.1452}}.

\bibitem[Kel11]{keller2011cluster}
Bernhard Keller, \emph{On cluster theory and quantum dilogarithm identities}, Representations of Algebras and Related Topics, Editors A. Skowronski and K. Yamagata, EMS Series of Congress Reports, European Mathematical Society, 2011, pp.~85--11.

\bibitem[KKKO18]{Kang2018}
S.-J. {Kang}, M.~{Kashiwara}, M.~{Kim}, and S.-j. {Oh}, \emph{{Monoidal categorification of cluster algebras}}, J. Amer. Math. Soc. \textbf{31} (2018), 349--426, \href {http://arxiv.org/abs/1412.8106} {\path{arXiv:1412.8106}}, \href {http://dx.doi.org/https://doi.org/10.1090/jams/895} {\path{doi:https://doi.org/10.1090/jams/895}}.

\bibitem[KO21]{kimura2017twist}
Yoshiyuki Kimura and Hironori Oya, \emph{Twist automorphisms on quantum unipotent cells and dual canonical bases}, International Mathematics Research Notices \textbf{2021} (2021), no.~9, 6772--6847, \href {http://arxiv.org/abs/1701.02268} {\path{arXiv:1701.02268}}.

\bibitem[McN21]{mcnamara2021cluster}
Peter~J McNamara, \emph{Cluster monomials are dual canonical}, \href {http://arxiv.org/abs/2112.04109} {\path{arXiv:2112.04109}}.

\bibitem[Qin14]{Qin12}
Fan Qin, \emph{t-analog of q-characters, bases of quantum cluster algebras, and a correction technique}, International Mathematics Research Notices \textbf{2014} (2014), no.~22, 6175--6232, \href {http://arxiv.org/abs/1207.6604} {\path{arXiv:1207.6604}}, \href {http://dx.doi.org/10.1093/imrn/rnt115} {\path{doi:10.1093/imrn/rnt115}}.

\bibitem[Qin17]{qin2017triangular}
\bysame, \emph{Triangular bases in quantum cluster algebras and monoidal categorification conjectures}, Duke Mathematical Journal \textbf{166} (2017), no.~12, 2337--2442, \href {http://arxiv.org/abs/1501.04085} {\path{arXiv:1501.04085}}.

\bibitem[Qin19]{qin2019compare}
\bysame, \emph{Compare triangular bases of acyclic quantum cluster algebras}, Transactions of the American Mathematical Society \textbf{372} (2019), no.~1, 485--501.

\bibitem[Qin20]{qin2020dual}
\bysame, \emph{Dual canonical bases and quantum cluster algebras}, \href {http://arxiv.org/abs/2003.13674} {\path{arXiv:2003.13674}}.

\bibitem[Qin21]{qin2021cluster}
\bysame, \emph{Cluster algebras and their bases}, proceedings of ICRA2020, to appear (2021), \href {http://arxiv.org/abs/2108.09279} {\path{arXiv:2108.09279}}.

\bibitem[Qin22]{qin2019bases}
\bysame, \emph{Bases for upper cluster algebras and tropical points}, Journal of the European Mathematical Society (2022), \href {http://arxiv.org/abs/1902.09507} {\path{arXiv:1902.09507}}.

\bibitem[Tra11]{Tran09}
Thao Tran, \emph{F-polynomials in quantum cluster algebras}, Algebr. Represent. Theory \textbf{14} (2011), no.~6, 1025--1061, \href {http://arxiv.org/abs/0904.3291v1} {\path{arXiv:0904.3291v1}}.

\end{thebibliography}
\def\cprime{$'$}
\providecommand{\bysame}{\leavevmode\hbox to3em{\hrulefill}\thinspace}
\providecommand{\MR}{\relax\ifhmode\unskip\space\fi MR }
\providecommand{\MRhref}[2]{%
  \href{http://www.ams.org/mathscinet-getitem?mr=#1}{#2}
}
\providecommand{\href}[2]{#2}

\end{document}